\documentclass[a4paper,11pt]{llncs}
\usepackage[a4paper,margin=2.0cm]{geometry}

\usepackage{setspace}
\newcommand{\parskipsize}{.4em}
\usepackage[shortlabels]{enumitem}
\usepackage{xcolor}
\usepackage{amsmath}
\usepackage{amssymb}
\usepackage{hyperref}
\usepackage{arydshln}
\usepackage{tikz}
\newtheorem{notation}[theorem]{Notation}
\renewcommand{\S}{\mathbb S}
\renewcommand{\P}{\mathbb P}
\newcommand{\F}{\mathbb F}
\DeclareMathOperator{\GL}{GL}
\DeclareMathOperator{\Aut}{Aut}
\DeclareMathOperator{\diag}{diag}
\DeclareMathOperator{\Gal}{Gal}
\begin{document}
\title{Counting the number of non-isotopic Taniguchi semifields}

\author{Faruk G\"olo\u{g}lu \inst{1} \and Lukas K\"olsch\inst{2}}
\institute{Charles University Prague \\ \email{farukgologlu@gmail.com} \and University of South Florida\\ \email{lukas.koelsch.math@gmail.com}}
\maketitle              % typeset the header of the contribution
\begin{abstract}
We investigate the isotopy question for Taniguchi semifields. We give a complete characterization when two Taniguchi semifields are isotopic. We further give  precise upper and lower bounds for the total number of non-isotopic Taniguchi semifields, proving that there are around $p^{m+s}$ non-isotopic Taniguchi semifields of order $p^{2m}$ where $s$ is the largest divisor of $m$ with $2s\neq m$. This result proves that the family of Taniguchi semifields is (asymptotically) the biggest known family of semifields of odd order. The key ingredient of the proofs is a technique to determine isotopy that uses group theory to exploit the existence of certain large subgroups of the autotopism group of a semifield. 
\end{abstract}
\section{Introduction}

A (finite) \textbf{semifield} $\mathbb{S} = (S,+,\circ)$ is a finite set $S$ equipped with two operations $(+,\circ)$
satisfying the following axioms. 
\begin{enumerate}
\item[(S1)] $(S,+)$ is a group.
\item[(S2)] For all $x,y,z \in S$,
\begin{itemize}
\item $x\circ (y+z) = x \circ y + x \circ z$,
\item $(x+y)\circ z = x \circ z + y \circ z$.
\end{itemize}
\item[(S3)] For all $x,y \in S$, $x \circ y = 0$ implies $x=0$ or $y=0$.
\item[(S4)] There exists $\epsilon \in S$ such that $x\circ \epsilon = x = \epsilon \circ x$.
\end{enumerate}
An algebraic object satisfying the first three of the above axioms is called 
a \textbf{pre-semifield}.

If $\P = (P,+,\circ)$ is a pre-semifield, then $(P,+)$ is 
an elementary abelian $p$-group \cite[p. 185]{Knuth65}, and $(P,+)$ can be viewed as
an $n$-dimensional $\F_{p}$-vector space $\F_p^n$. A pre-semifield $\P = (\F_p^n,+,\circ)$ 
can be converted to a semifield $\S = (\F_p^n,+,\ast)$ using {\em Kaplansky's trick} (see e.g.~\cite[Section 1.1]{lavrauw2011finite}).

Two pre-semifields $\P_1 = (\F_p^n,+,\circ_1)$ and $\P_2 = (\F_p^n,+,\circ_2)$ 
are said to be \textbf{isotopic} if there exist $\F_{p}$-linear bijections 
$L,M$ and $N$ of $\F_p^n$ satisfying
\[
N(x \circ_1 y) = L(x) \circ_2 M(y).
\]
Such a triple $\gamma = (N,L,M)$ is called an \textbf{isotopism} between $\P_1$ and $\P_2$. 
Isotopisms between a pre-semifield $\P$ and itself are called \textbf{autotopisms}. Isotopy of pre-semifields is an equivalence relation and 
the pre-semifield $\P$ and the corresponding semifield $\S$ constructed 
by Kaplansky's trick are always isotopic.

Research on semifields started more than 100 years ago with the work of Dickson~\cite{dickson1906commutative}. Over time, they received much attention due to their connections to several different areas. Firstly, every semifield coordinatizes a projective plane and different semifields coordinatize
isomorphic planes if and only if they are isotopic 
(\cite{Albert60}, see \cite[Section 3]{Knuth65} for a detailed
treatment). More recently, semifields have been the center of much attention since they are equivalent to Maximum Rank Distance codes with certain parameters (see e.g.~\cite{sheekey}) and can be used to construct other combinatorial structures like relative difference sets~(see \cite{pott2014semifields}).

Deciding whether given (pre-)semifields are isotopic or not is generally a very difficult question, and finding effective ways to prove non-isotopy of semifields is considered a major open question (see e.g. \cite[p. 936]{KW}). Most results on the isotopy of semifields are based on isotopy invariants like the nuclei, however it is well known that potentially many non-isotopic semifields can have the same nuclei, and having more precise tools is desirable. In~\cite{comm_arxiv}, the authors developed a technique to settle the isotopy question for a specific family of commutative semifields.  In this work, we will focus on the (non-commutative) Taniguchi semifields introduced in~\cite{tani_semi} by extending the methods introduced in~\cite{comm_arxiv}.

Note that many constructions of semifields also yield corresponding constructions 
of almost perfect nonlinear (APN) functions which play a big role in the design of 
block ciphers for cryptography. This is also the case with the Taniguchi semifields.
However, the construction used by Taniguchi for the semifields, by design, are 
more complicated than that for the Taniguchi APN functions. The equivalence question 
for the Taniguchi APN functions was recently solved in~\cite{kaspers_phd,kasperszhou} 
using a more elementary, but a very technical approach compared to the techniques 
we use. A variant of the approach we introduce here yields a much shorter proof for 
the equivalence problem for the Taniguchi APN functions. It seems that the isotopy 
question of the Taniguchi semifields is more complex than the equivalence question 
for the Taniguchi APN functions. 

In Section~\ref{sec:setup}, we give the basic definitions that are important for our problem and give some simple general results. In Section~\ref{sec:group}, we introduce the group theoretic techniques that are key to later sections. In Section~\ref{sec:isotopy}, we investigate when two Taniguchi pre-semifields are isotopic. A complete characterization is given in Theorem~\ref{thm:tani_isotopy}. Section~\ref{sec:counts} deals with giving a count of non-isotopic Taniguchi pre-semifields, with precise bounds given in Theorem~\ref{thm_count}. The last section compares these results to similar results for other semifields, in particular to semifields constructed via skew-polynomial rings (or cyclic semifields).

\section{The setup} \label{sec:setup}

The Taniguchi pre-semifields are defined in~\cite{tani_semi} on $\F_{p^n}\cong \F_{p^m} \times \F_{p^m}$ with $n=2m$ via the pre-semifield multiplication
\[(x,y) \ast (u,v) = ((x^qu+\alpha xu^q)^{q^2}-a(x^qv-\alpha u^qy)^q-b(y^qv+\alpha yv^q),xv+yu),\]
where 
\begin{itemize}
\item $q=p^k$ for some $1 \leq k \leq m-1$, 
\item $-\alpha$ is not a $(q-1)$-st power, and 
\item the projective polynomial $P_{q,a,b}(x) = x^{q+1}+ax+b$ has no roots in $\F_{p^m}$. 
\end{itemize}
In this paper, we will instead use a different, isotopic representation of the Taniguchi pre-semifield. This representation arises after taking $x,u$ to the $\overline{q}^2$-th power, where $\overline{q}=p^{m-k}$, and then taking the second component to the $q^2$-th power:
\begin{equation}
 (x,y) \circ (u,v) = (x^qu+\alpha^{q^2} xu^q-a(xv^q-\alpha^q uy^q)-b(y^qv+\alpha yv^q),xv^{q^2}+y^{q^2}u).
\label{eq:tanidef}
\end{equation}
The benefit of this representation is that both components of the operations employ only one nontrivial field automorphism (namely, $x \mapsto x^q$ in the first, and $x \mapsto x^{q^2}$ in the second component), which gives in particular more structure to certain autotopisms as we will see later.

If $a\neq 0$ we can always find an isotopic Taniguchi pre-semifield with the parameter $a=1$ by using the transformation $y \mapsto \delta y$ and $v \mapsto \delta v$ for a suitable $\delta \in \F_{p^m}^*$. We thus only have to distinguish the cases $a=0$ and $a=1$ when discussing the isotopy question. 
{We will denote the Taniguchi pre-semifield on $\F_{p^n}\cong \F_{p^m} \times \F_{p^m}$ by $T(q,\alpha,a,b)$, where the value of $m$ is fixed and taken from context.}

We also exclude the case $k=m/2$ since in this case $q^2\equiv 1 \pmod {p^m-1}$ which is a special case that requires slightly different methods. Also observe that these pre-semifields are already contained in a family of Bierbrauer~\cite{bbsf,bbeven}, so we believe that it makes sense to exclude them from our treatment here.

It is possible to discern some isotopisms immediately:
\begin{proposition}  \label{prop:minus}
Let $q = p^{k}$ and $\overline{q} = p^{m-k}$.
Each Taniguchi pre-semifield $T(q,\alpha,a,b)=\P_1$ is isotopic to 
another Taniguchi pre-semifield 
$T(\overline{q},1/\alpha^{q^2},a(\alpha^{q-1}/b),\alpha^{q^2-1}/b)=\P_2$. 
\end{proposition}

\begin{proof}

We first perform a change of variables $x \leftrightarrow y$, $u \leftrightarrow v$ on $\P_1$ (which clearly preserves isotopy). The result is 
\[(x,y) *_1 (u,v)  = (y^qv+\alpha^{q^2}yv^q-a(yu^q-\alpha^q x^qv)-b(x^qu+\alpha xu^q),yu^{q^2}+x^{q^2}v).\]
We take the second component to the power $\overline{q}$ and then $x,y,u,v$ to the power $\overline{q}$ as well. The result is
\begin{align*}
(x,y) *_2 (u,v)  =& (yv^{\overline{q}}+\alpha^{q^2} y^{\overline{q}}v-a(y^{\overline{q}}u-\alpha^q xv^{\overline{q}})-b(xu^{\overline{q}}+\alpha x^{\overline{q}}u),y^{\overline{q}^2}u+xv^{\overline{q}^2}) \\
=&(\alpha^{q^2} ((1/\alpha^{q^2}) yv^{\overline{q}}+y^{\overline{q}}v)-a\alpha^q((1/\alpha^q) y^{\overline{q}}u-xv^{\overline{q}})-b\alpha((1/\alpha)xu^{\overline{q}}+ x^{\overline{q}}u),\\&y^{\overline{q}^2}u+xv^{\overline{q}^2}).
\end{align*}
Now we can divide the first component by $-b\alpha$, which yields
\begin{align*}
(x,y) *_3 (u,v)  =&(x^{\overline{q}}u+(1/\alpha)xu^{\overline{q}}-a(\alpha^{q-1}/b)(xv^{\overline{q}}-(1/\alpha^q) y^{\overline{q}}u)- (\alpha^{q^2-1}/b) (y^{\overline{q}}v+(1/\alpha^{q^2}) yv^{\overline{q}}),\\&xv^{\overline{q}^2}+y^{\overline{q}^2}u),
\end{align*}
proving the desired isotopy.
\qed
\end{proof}

With Proposition~\ref{prop:minus} it suffices to consider coefficients $q=p^k$ with $k < m/2$ when tackling the isotopy question (recall that we exclude the case $k=m/2$).

\section{Group theoretic preliminaries}\label{sec:group}
We now introduce the machinery of the technique to determine isotopy. The ideas are based on an approach developed by the authors in \cite{comm_arxiv} for a family of commutative semifields.

We denote the set of all autotopisms of a pre-semifield $\P$ by $\Aut(\P)$. It is 
easy to check that $\Aut(\P)$ is a group under component-wise composition, i.e., 
$(N_1,L_1,M_1) \circ (N_2,L_2,M_2) = (N_1 \circ N_2, L_1 \circ L_2, M_1 \circ M_2)$. 
We will often view $\Aut(\P)$ as a subgroup of $\GL(\F_{p^n})^3 \cong \GL(\F_{p^m} \times \F_{p^m})^3 \cong \GL(n,\F_p)^3$.
Our approach is based on the following simple and well-known result (see e.g. \cite{comm_arxiv}).

\begin{lemma} \label{lem:conj_groups}
Let $\P_1=(\F_p^n,+,\circ_1)$, $\P_2=(\F_p^n,+,\circ_2)$ be isotopic
pre-semifields via the isotopism $\gamma \in \GL(\F_{p^n})^3$. Then 
$\gamma^{-1}\Aut(\P_2)\gamma = \Aut(\P_1)$. 
\end{lemma}

The key fact that we will use is that the autotopism groups of the Taniguchi pre-semifields have a large and easily identifiable subgroup. We introduce some notations:

We write $\F_{p}$-linear mappings $L$ from $\F_{p^n}$ to itself as $2 \times 2$ matrices
of       $\F_{p}$-linear mappings     from $\F_{p^m}$ to itself.
That is,
\[
L=\begin{pmatrix}
	L_1 & L_2 \\
	L_3 & L_4
\end{pmatrix}, \textrm{ for } L_i \colon \F_{p^m} \rightarrow \F_{p^m}. 
\]
We call the constituent functions $L_1,\dots,L_4$ of $L$ \textit{subfunctions of} $L$.
Set 
\[
\gamma_r=(N_r,L_r,M_r)\in \GL(\F_{p^n})^3  \textrm{ with } 
N_r=\begin{pmatrix}
	m_{r^{q+1}} & 0 \\
	0 & m_{r^{q^2+1}}
\end{pmatrix}, \quad 
L_r=M_r=\begin{pmatrix}
	m_{r} & 0 \\
	0 & m_{r}
\end{pmatrix}, 
\]
where $m_r$ denotes  multiplication with the finite field element $r \in \F_{p^m}^*$. 
For simplicity, we write these diagonal matrices also in the form $\diag(m_{r},m_{r})$, 
so 
\[
\gamma_r = (\diag(m_{r^{q+1}},m_{r^{q^2+1}}),\diag(m_{r},m_{r}),\diag(m_{r},m_{r})).
\]

We fix some further notation that we will use from now on:
\begin{notation}
\begin{itemize}
\item[]
\item Let $p$ be a prime.
\item Set $q=p^k$ with $k<m/2$ and $\overline{q}=p^{m-k}$.
\item Define the cyclic group 
\[
	Z^{(q)} = \{ \gamma_r : r \in \F_{p^m}^*\} \leq \GL(\F_{p^n})^3
\] of order $p^m-1$. It is easy to see (Lemma~\ref{lem_first} below) that $Z^{(q)}\leq \Aut(\P)$.
\item Let $p'$ be a $p$-primitive divisor of $p^m-1$, i.e. $p'|p^m-1$ and $p'\nmid p^{k'}-1$ for $k'<m$. Such a prime $p'$ always exists if 
$m > 2$ and $(p,m) \neq (2,6)$ by Zsigmondy's Theorem (see e.g.~\cite[Chapter IX., Theorem 8.3.]{HuppertII}).  We thus always stipulate $m>2$ 
and $(p,m)\neq (2,6)$ from now on.
\item Let $R$ be the unique Sylow $p'$-subgroup of $\F_{p^m}^*$.
\item Define
\[
Z_R^{(q)} = \{ \gamma_r \colon r \in R\},
\]
which is the unique Sylow $p'$-subgroup of $Z^{(q)}$ with $|R|$ elements. 
\item For a Taniguchi pre-semifield $\P=T(q,\alpha,a,b)$, denote by 
\[
C_{q,\alpha,a,b} = C_{\Aut(\P)}(Z_R^{(q)}), 
\]
the centralizer of $Z_R^{(q)}$ in $\Aut(\P)$.
\item Define
\[
	S = \{\diag(m_r,m_r) \colon r \in \F_{p^m}^*\},
\]
and
\[
	S_R = \{\diag(m_r,m_r) \colon r \in R\}.
\]
\end{itemize}
\end{notation}
Observe that the condition $m>2$ that is necessary to work with a Zsygmondy prime is actually not restrictive since for $m=2$ the only admissible value for $q$ is $q=p=p^{m/2}$ which is precisely the case we exclude anyway.

The crucial fact for our technique is that
$\gamma_r \in \Aut(\P)$ for all $r \in \F_{p^m}^*$ when $\P$ is a Taniguchi
pre-semifield $T(q,\alpha,a,b)$ for arbitrary $\alpha,a,b$, which can be directly verified using Eq.~\eqref{eq:tanidef}:

\begin{lemma}\label{lem_first}
Let $T(q,\alpha,a,b)=\P$ be a Taniguchi pre-semifield on $\F_{p^n}$ with $n=2m$. Then 
$Z^{(q)} \le \Aut(\P)$.
\end{lemma}

The key result that enables us to settle the question of isotopy for Taniguchi semifields is a slight adaptation from~\cite[Theorem 5.10.]{comm_arxiv}, which deals with certain \emph{commutative} pre-semifields. In some sense, the result we present here is an adaptation of the one from~\cite{comm_arxiv} to a non-commutative semifield. 

\begin{lemma}[{\cite[Lemma 5.7.]{comm_arxiv}}] \label{lem:yoshiara}
 Let $N_{\GL(\F_{p^n})}(S_R)$, $N_{\GL(\F_{p^n})}(S)$ and $C_{\GL(\F_{p^n})}(S_R)$, $C_{\GL(\F_{p^n})}(S)$ be the normalizers and the centralizers of 
 $S_R$ and $S$ in $\GL(\F_{p^n})$. Then
	\begin{enumerate}[(a)]
		\item $N_{\GL(\F_{p^n})}(S_R)=N_{\GL(\F_{p^n})}(S) = \left\{\begin{pmatrix}
			m_{c_1} \tau & m_{c_2} \tau \\
			m_{c_3} \tau&m_{c_4} \tau
		\end{pmatrix} \colon c_1,c_2,c_3,c_4 \in \F_{p^m}^*, \tau \in \Gal(\F_{p^m}/\F_p)\right\}\cap \GL(\F_{p^n})$,
		\item $C_{\GL(\F_{p^n})}(S_R)=C_{\GL(\F_{p^n})}(S) = \left\{\begin{pmatrix}
			m_{c_1}  & m_{c_2}  \\
			m_{c_3}&m_{c_4} 
		\end{pmatrix} \colon c_1,c_2,c_3,c_4 \in \F_{p^m}^*\right\}\cap \GL(\F_{p^n})$.
	\end{enumerate}
\end{lemma}

The following is an analogue of~\cite[Lemma 5.8.]{comm_arxiv}.
\begin{lemma}\label{lem:zpsylow}
Let $\P=T(q,\alpha,a,b)$  be a Taniguchi pre-semifield. Assume that 
$C_{q,\alpha,a,b}$ contains $Z^{(q)}$ as an index $I$ subgroup such that $p'$ does not divide $I$. 
Then $Z_R^{(q)}$ is a Sylow $p'$-subgroup of $\Aut(\P)$.
\end{lemma}
\begin{proof}

Let $T$ be a Sylow $p'$-subgroup of $\Aut(\P)$ that contains the 
$p'$-group $Z_R^{(q)}$. $T$ itself is (by Sylow's Theorem) contained in
a Sylow $p'$-subgroup of $\GL(\F_{p^n})^3$, say $U$. In particular, $T$ is abelian since all Sylow $p'$-subgroups of $\GL(\F_{p^n})^3$ are abelian, see~\cite[Proof of Lemma 5.8.]{comm_arxiv}.
This implies that $T$ is a subgroup of the centralizer 
$C_{q,\alpha,a,b}$ of $Z_R^{(q)}$ in $\Aut(\P)$. By assumption, $Z^{(q)}$ is an index 
$I$ subgroup of $C_{q,\alpha,a,b}$ and $p'$ does not divide $I$.
Moreover, $Z_R^{(q)}$ is a Sylow $p'$-subgroup of $Z^{(q)}$ and therefore
$p'\nmid [Z^{(q)}:Z_R^{(q)}] = I_1$. Let $[T:Z_R^{(q)}] = I_2 = p'^h$ for $h \ge 0$, 
since both are $p'$-groups. Since $I_2|I_1 I$,
and $p'\nmid I_1 I$, we must have $p'\nmid I_2$ and $I_2 = 1$. Thus,
$Z_R^{(q)}=T$ and $Z_R^{(q)}$ is a Sylow $p'$-subgroup of $\Aut(\P)$ as claimed.
\qed
\end{proof}

The following theorem is the main result that enables us to solve the isotopy question. It states that if two Taniguchi pre-semifields are isotopic (and a certain condition is satisfied), then there must exist an isotopism of a very simple form. Note that this does not prove that \emph{all} isotopisms necessarily have this structure. 

\begin{theorem} \label{thm:equivalence}
Let $\P_1=T(q_1,\alpha,a,b)$ and $\P_2=T(q_2,\alpha',a',b')$ be Taniguchi pre-semifields such that $0<k_1<m/2$ and $0<k_2\leq m/2$.
Assume that 
	\begin{equation}
		C_{q_1,\alpha,a,b} \text{ contains }Z^{(q_1)} \text{ as an index } I \text{ subgroup such that }p'\text{ does not divide }I.
	\tag{C}\label{eq:condition}
	\end{equation}
If $\P_1,\P_2$ are %(strongly)
 isotopic, then there exists an %(strong)
 isotopism $\gamma=(N,L,M) \in \GL(\F_{p^n})^3$ such that all non-zero subfunctions of $L,M$ are monomials. Moreover, all non-zero subfuncions of $L$ and $M$ have the same degree. (The degree of the subfunctions of $L$ could be different than the degree of the subfunctions of $M$).
\end{theorem}

\begin{proof}
	
%Set 
%\begin{align*}
	%(x,y)\ast_1(u,v)&=(x^{q_1}u+\alpha_1^{q_1^2} xu^q_1-a_1(xv^q_1-\alpha_1^{q_1} uy^{q_1})-b_1(y^{q_1}v+\alpha_1 yv^{q_1}),xv^{q_1^2}+y^{q_1^2}u), \textrm{ and}\\
	%(x,y)\ast_2(u,v)&=(x^{q_2}u+\alpha_2^{q_1^2} xu^q_2-a_2(xv^q_2-\alpha_1^{q_2} uy^{q_2})-b_2(y^{q_2}v+\alpha_2 yv^{q_2}),xv^{q_2^2}+y^{q_2^2}u).
%\end{align*}
By Lemma \ref{lem_first}, we have $Z_R^{(q_1)} \le \Aut(\P_1)$
and $Z_R^{(q_2)} \le \Aut(\P_2)$.
Assume $\P_1$ and $\P_2$ are isotopic via the isotopism $\delta \in \GL(\F_{p^n})^3$ 
that maps $\P_1$ to $\P_2$. 
Then $\delta^{-1} \Aut(\P_2) \delta = \Aut(\P_1)$ by Lemma~\ref{lem:conj_groups}. 
Observe that $|\delta^{-1}Z_R^{(q_2)} \delta|=|R|=|Z_R^{(q_1)}|$, 
so $Z_R^{(q_1)}$ and $\delta^{-1}Z_R^{(q_2)}\delta$ are 
Sylow $p'$-subgroups of  $\Aut(\P_1)$ by Lemma~\ref{lem:zpsylow} as long as 
Condition~\eqref{eq:condition} holds. In particular, these two subgroups are 
conjugate in $\Aut(\P_1)$ by Sylow's theorem, i.e., there exists a 
$\lambda \in \Aut(\P_1)$ such that 
\begin{equation}
\label{eq:conjugated}
\lambda^{-1} \delta^{-1} Z_R^{(q_2)} \delta \lambda = (\delta \lambda)^{-1} Z_R^{(q_2)} (\delta\lambda) = Z_R^{(q_1)}.
\end{equation}

Set $\gamma= (N,L,M)=\delta \lambda$. Note that $\gamma$ is an isotopism between 
$\P_1$ and $\P_2$ since $\lambda \in \Aut(\P_1)$. Eq.~\eqref{eq:conjugated} then 
immediately implies that 
\begin{align*}
\diag(m_{r^{q_2+1}},m_{r^{q_2^2+1}}) N &= N\diag(m_{s^{q_1+1}},m_{s^{q_1^2+1}}) \\
\diag(m_{r},m_{r}) L &= L \diag(m_{s},m_{s}) \\
\diag(m_{r},m_{r}) M &= M \diag(m_{s},m_{s}) 
\end{align*}
for all $r \in R$ and $s=\pi(r)$ where $\pi \colon R \rightarrow R$ is a permutation. 
In particular, $L$ and $M$ are in the normalizer of 
$S_R = \{\diag(m_{a},m_{a}) \colon a \in R\}$. 
By Lemma~\ref{lem:yoshiara}, all of the four subfunctions of $L$ and $M$  are zero or 
monomials of the same degree.
\qed
	\end{proof}

We will now systematically investigate isotopisms $(N,L,M)$ where the 
subfunctions of $L,M$ are monomials or zero. We want to emphasize that 
without this simplification, a treatment of the isotopy question is 
\emph{very} complicated, whereas the calculations we will do, while 
still technical, are much easier to handle. We also note that the 
remaining question on the crucial Condition~\eqref{eq:condition} is 
naturally answered along the way and does not require much additional work.

\section{Settling the isotopy question for Taniguchi semifields} \label{sec:isotopy}

We apply Theorem~\ref{thm:equivalence}. First, we achieve some further strong restrictions.

\begin{proposition}\label{prop:firststep}
Let $q_1=p^{k_1}$, $q_2=p^{k_2}$, 
\begin{align*}
\P_1&=T(q_1,\alpha,a,b)=(\F_{p^m} \times \F_{p^m},+,\circ_1), \textrm{ and},\\
\P_2&=T(q_2,\alpha',a',b')=(\F_{p^m} \times \F_{p^m},+,\circ_2)
\end{align*}
be Taniguchi pre-semifields such that $0< k_1<m/2$, $0<k_2\leq m/2$. Further, let $(N,L,M)$ be an isotopism between $\P_1,\P_2$ such that all non-zero subfunctions of $L,M$ are monomials of the same degree. Then $N_2,N_3,L_2,L_3,M_2,M_3=0$, all other subfunctions are monomials of the same degree, and $k_1=k_2$.
\end{proposition} 
\begin{proof}
	Say the degree of the non-zero subfunctions of $L$ and $M$ is $p^{t_2}$ and $p^{t_3}$, respectively. Then, for all $(x,y),(u,v)\in\F_{p^m}^2$,
\begin{align*}
L(x,y) \circ_2 M(u,v) &= (a_2x^{p^{t_2}}+b_2y^{p^{t_2}},c_2x^{p^{t_2}}+d_2y^{p^{t_2}}) \circ_2 
                        (a_3u^{p^{t_3}}+b_3v^{p^{t_3}},c_3u^{p^{t_3}}+d_3v^{p^{t_3}})\\
	                   &= (h_1(x,y,u,v), h_2(x,y,u,v))
\end{align*}
	for some $a_2,b_2,c_2,d_2,a_3,b_3,c_3,d_3 \in \F_{p^m}$. 
	
	Hence, 
				\begin{align}
		h_2&= (a_2x^{p^{t_2}}+b_2y^{p^{t_2}})(c_3u^{p^{t_3}}+d_3v^{p^{t_3}})^{q_2^2} +(c_2x^{p^{t_2}}+d_2y^{p^{t_2}})^{q_2^2}(a_3u^{p^{t_3}}+b_3v^{p^{t_3}}).\label{eq:degrees3}
	\end{align}
	
	 We also have
	\begin{align*}
		N((x,y)\circ_1(u,v)) = (\ast, &N_{3}(x^{q_1}u+\alpha^{q_1^2} xu^{q_1}-a(xv^{q_1}-\alpha^{q_1} uy^{q_1})-b(y^{q_1}v+\alpha yv^{q_1}))\\
		                              &+N_{4}(xv^{q_1^2}+y^{q_1^2}u)).
	\end{align*}
	 Set $N((x,y)\circ_1(u,v)) = L(x,y)\circ_2 M(u,v)$. Let us assume $N_3 \neq 0$, i.e. the second component contains a term 
	\begin{equation}
		c(x^{q_1}u+\alpha^{q_1^2} xu^{q_1}-a(xv^{q_1}-\alpha^{q_1} uy^{q_1})-b(y^{q_1}v+\alpha yv^{q_1}))^{p^t}.
	\label{eq:term1}
	\end{equation}

		Note that none of these terms can be 
	 canceled out by $N_4(xv^{q_1^2}+y^{q_1^2}u)$. Thus, those monomials also have to occur in $h_2$. Let us consider the monomials $x^{p^{k_1+t}}u^{p^t}$, $x^{p^{t}}u^{p^{k_1+t}}$. Those appear in $h_2$ if and only if $t_2=k_1+t$, $t_3+2k_2=t$, $t_2+2k_2=t$, $t_3=k_1+t$, $k_1 \equiv -2k_2 \pmod m$, or $t_2=t_3=t$, $k_1 \equiv 2k_2 \pmod m$. In both cases we get $t_2=t_3$ and $t$ is uniquely determined by $t_2$, so $N_3$ is a monomial. In order for all monomials in Eq.~\eqref{eq:term1} to occur in $h_2$, it is necessary that $a_2a_3b_2b_3c_2c_3d_2d_3\neq 0$, but then $h_2$ will also contain the terms $y^{p^{t_2}}u^{p^{t_3+2k_2}}$ and $y^{p^{t_2+2k_2}}u^{p^{t_3}}$ which cannot both occur in the second component of $N((x,y)\circ_1(u,v))$ (again, by the choice of $t_2,t_3$ and the conditions on $k_1,k_2$ it is impossible that those terms are covered by $N_4(xv^{q_1^2}+y^{q_1^2}u)$). We infer $N_3=0$.
	
	Thus 
	\[N((x,y)\circ_1(u,v)) =(\ast,N_{4}(xv^{q_1^2}+y^{q_1^2}u)).\]
	Comparing with Eq.~\eqref{eq:degrees3} immediately yields that $N_4$ is a monomial, say of degree $p^t$, and either $t=t_2=t_3$, $q_1=q_2$, $b_2=c_2=b_3=c_3=0$ or $t_2+2k_2=t$, $t_3=t+2k_1$, $t+2k_1=t_2$, $t=t_3+2k_2$. The second case leads to $t_2=t_3$ and $k_1\equiv -k_2 \pmod m$, which is by our condition $ 0<k_1<m/2$, $0<k_2\leq m/2$ impossible. So $q_1=q_2$, $t=t_2=t_3$, $b_2=c_2=b_3=c_3=0$ (implying $L_2=L_3=M_3=M_3=0$).
	
Let us now check the first component. We have
\begin{align*}
N((x,y)\circ_1(u,v)) = (N_{1}( & x^{q_1}u+\alpha^{q_1^2} xu^{q_1} \\
                                &-a_1(xv^{q_1}-\alpha^{q_1} uy^{q_1})\\
																&-b_1(y^{q_1}v+\alpha yv^{q_1}) )\\
											  +N_{2}&(xv^{q_1^2}+y^{q_1^2}u),\ast)
\end{align*}
	and
\begin{align*}
		h_2 = &C_1x^{p^{k_1+t_2}}u^{p^{t_2}}+\alpha'^{q_1^2}C_2 x^{p^{t_2}}u^{p^{k_1+t_2}}\\
		      &-a'(C_3x^{p^{t_2}}v^{p^{k_1+t_2}}-\alpha'^{q_1}C_4 u^{p^{t_2}}y^{p^{k_1+t_2}})\\
					&-b'(C_5y^{p^{k_1+t_2}}v^{p^{t_2}}+\alpha' C_6y^{p^{t_2}}v^{p^{k_1+t_2}})
	\end{align*}	
	for non-zero coefficients $C_1,\dots,C_6$.
	A comparison of degrees immediately shows $N_2=0$ and that $N_1$ is a monomial of degree $p^t$ as desired.
	\qed
\end{proof}

In the next step, we determine the remaining subfunctions. This also immediately verifies Condition~\eqref{eq:condition}.

\begin{proposition} \label{prop:2}
	Let $q_1=p^{k_1}$, $q_2=p^{k_2}$,
\begin{align*}
\P_1&=T(q_1,\alpha,a,b)=(\F_{p^m} \times \F_{p^m},+,\circ_1), \textrm{ and},\\
\P_2&=T(q_2,\alpha',a',b')=(\F_{p^m} \times \F_{p^m},+,\circ_2)
\end{align*}
be Taniguchi pre-semifields such that $0< k_1<m/2$ and $0<k_2\leq m/2$.  Further, let $(N,L,M)$ be an isotopism between $\P_1,\P_2$ such that all non-zero subfunctions of $M,N$ are monomials of the same degree. Then
	\begin{itemize}
		\item $a=a'=0$, $\alpha^{p^t}/\alpha'$  is a $(q-1)$-st power  in $\F_{p^m}^*$ and $b'^{p^t}/b$ is a $(q+1)$-st power in $\F_{p^m}^*$ for some $0\leq t \leq m-1$, or
		\item $a=a'=1$, $\alpha^{p^t}/\alpha'$  is a $(q-1)$-st power in $\F_{p^m}^*$ and $b=b'^{p^t}$ for some $0\leq t \leq m-1$.
	\end{itemize}
	
	Moreover, \[|C_{q_1,\alpha,a,b}|=\begin{cases}
	(p^{\gcd(k,m)}-1)(p^m-1) & \text{ if }a \neq 0 \\
	(p^{\gcd(k,m)}-1)(p^m-1)\cdot \gcd(p^m-1,p^k+1) & \text{ if }a=0.
\end{cases}\]
\end{proposition}
\begin{proof}
	From Proposition~\ref{prop:firststep}, we infer that $L_2,L_3,N_2,N_3,M_2,M_3=0$, all other subfunctions are monomials of the same degree $p^t$, and $q_1=q_2=:p^k=q$. 
	
	Set $N_1=a_1x^{p^{t}}$, $N_4=d_1x^{p^{t}}$. Then
	\begin{equation}
N((x,y)\circ_1(u,v))=(a_1(x^qu+\alpha^{q^2}xu^q-a(xv^q-\alpha^q y^qu)-b(y^qv+\alpha yv^q))^{p^{t}},d_1(xv^{q^2}+y^{q^2}u)^{p^{t}}).
\label{eq:tan_RHS}
\end{equation}
Likewise, the subfunctions of $L$ and $M$ are monomials of degree $p^t$, so
	\[L((x,y))\circ_2 M((u,v))=(a_2x^{p^{t}},d_2y^{p^{t}})\circ_2(a_3u^{p^{t}},d_3v^{p^{t}}).\]
	
Thus
\[
L(x,y)\circ_2 M(u,v) = (A_1(x,y,u,v),A_2(x,y,u,v))
\]
where
\begin{align*}
A_1(x,y,u,v) =& (a_2x)^{p^{t+k}}(a_3u)^{p^{t}}+\alpha'^{q^2}(a_2x)^{p^{t}}(a_3u)^{p^{t+k}}\\
              & -a'\left((a_2x)^{p^{t}}(d_3v)^{p^{t+k}}-\alpha'^q (a_3u)^{p^{t}}(d_2y)^{p^{t+k}}\right)\\
			  & -b'\left((d_2y)^{p^{t+k}}(d_3v)^{p^{t}}+\alpha' (d_2y)^{p^{t}}(d_3v)^{p^{t+k}}\right),
\end{align*}
and
\[
A_2(x,y,u,v) = (a_2x)^{p^{t}}(d_3v)^{p^{t+2k}}+(a_3u)^{p^{t}}(d_2y)^{p^{t+2k}}.
\]

A comparison with Eq.~\eqref{eq:tan_RHS} yields for all possible terms $(x^qu)^{p^t}$, $(xu^q)^{p^t}$, $(xv^q)^{p^t}$, $(y^qu)^{p^t}$, $(y^qv)^{p^t}$, $(yv^q)^{p^t}$ in the first component and the two terms in the second component the following 8 equations:

 \noindent\begin{minipage}{0.5\textwidth}
\begin{align}
	a_1&=(a_2^{q}a_3)^{p^t} \label{eq:t_l1}\\
	a_1 \alpha^{q^2+p^t} &= \alpha'^{q^2}(a_2a_3^q)^{p^{t}}\label{eq:t_l2}\\	
	a_1a &= a'(a_2d_3^q)^{p^t}\label{eq:t_l3}\\
	a_1a\alpha^{q+p^t} &=a'\alpha'^q(a_3d_2^q)^{p^t}\label{eq:t_l4} 
\end{align}
    \end{minipage}%
    \begin{minipage}{0.5\textwidth}
	\begin{align}
		a_1b^{p^t} &= b'(d_2^qd_3)^{p^t}\label{eq:t_r1} \\
		a_1b^{p^t}\alpha^{p^t} &= b'\alpha' (d_2d_3^q)^{p^t} \label{eq:t_r2}\\
		d_1 &= (a_2d_3^{q^2})^{p^{t}} \label{eq:t_r3}\\
		d_1 &= (a_3d_2^{q^2})^{p^{t}}.\label{eq:t_r4}
	\end{align}
    \end{minipage}\vskip1em
	Eq.~\eqref{eq:t_l3} can only be satisfied if $a=a'=0$ or $a=a'=1$, so we only need to consider these two cases.
	
 Substituting Eq.~\eqref{eq:t_l1} into Eq.~\eqref{eq:t_l2} yields $(a_2^qa_3)^{p^{t}} (\alpha^{p^t}/\alpha')^{q^2} = (a_2a_3^q)^{p^t}$ which leads to 
	\begin{equation}
			a_2^{q-1}(\alpha^{p^t}/\alpha')^{p^{2k-t}} = a_3^{q-1}.
	\label{eq:t_q1}
	\end{equation}

	In particular, $\alpha^{p^t}/\alpha'$ must be a $(q-1)$-st power. We set $a_3 = a_2 \gamma$, where $\gamma^{q-1} = (\alpha^{p^t}/\alpha')^{p^{2k-t}}$. 
	Similarly, substituting Eq.~\eqref{eq:t_r1} into Eq.~\eqref{eq:t_r2} yields
		\[d_2^{q-1}(\alpha^{p^t}/\alpha')^{p^{m-t}} = d_3^{q-1},\]
		and we set $d_3=d_2 \gamma_2$ where $\gamma_2^{q-1} = (\alpha^{p^t}/\alpha')^{p^{m-t}}$. 
		%Note that both $\gamma,\gamma_2$ are determined by $\alpha,\alpha'$ only up to a multiplication with an element $\omega \in \F_{p^{\gcd(k,m)}}$. 
		Comparing now Eq.~\eqref{eq:t_r3} with Eq.~\eqref{eq:t_r4} gives $a_2d_2^{q^2} \gamma_2^{q^2} = a_2d_2^{q^2} \gamma$, that is $\gamma=\gamma_2^{q^2}$. 
		A comparison between Eq.~\eqref{eq:t_l1} and  Eq.~\eqref{eq:t_r1} yields 
		\begin{align}
				(b'^{p^{m-t}}/b) & = (a_2/d_2)^{q+1} \gamma/\gamma_2 = (a_2/d_2)^{q+1} \gamma_2^{q^2-1} \nonumber\\
				 & = 	(a_2/d_2)^{q+1} (\alpha^{p^t}/\alpha')^{p^{m-t}\cdot (q+1)}.	\label{eq:t_condition}
		\end{align}

		Thus $(b'^{p^{m-t}}/b)$ must also be a $(q+1)$-st power; in other words $b$ and $b'^{p^{m-t}}$ have to be in the same coset of the subgroup of $(q+1)$-st powers in $\F_{p^m}^*$. 

			We now consider the case $a=a'=0$. Then Eq.~\eqref{eq:t_l3} and Eq.~\eqref{eq:t_l4} always hold and the conditions we have gathered so far cover all equations. We can thus find an isotopism between $T(q,\alpha,0,b)$ and $T(q,\alpha',0,b')$ if and only if $\alpha^{p^t}/\alpha'$ is a $(q-1)$-st power and $(b'^{p^{t}}/b)$ is a $(q+1)$-st power for some $t \in \mathbb{N}$.
			
			Now consider the case $a=a'=1$. Of course, all previously derived constraints still apply, and Eq.~\eqref{eq:t_l3} and Eq.~\eqref{eq:t_l4} give two additional conditions. We first rewrite Eq.~\eqref{eq:t_l3} with Eq.~\eqref{eq:t_l1}. The result is $a_2d_3^q=a_2^qa_3$ and, using Eq.~\eqref{eq:t_q1}, we get
			\[d_3^q = \frac{a_3^q}{\gamma^{q-1}}=a_3^q(\alpha'/\alpha^{p^t})^{p^{2k-t}}.\]
			Similarly, rewriting Eq.~\eqref{eq:t_l4} with Eq.~\eqref{eq:t_l1} yields
			\[a_2^q \alpha^{p^k}/\alpha'^{p^{k-t}} = d_2^q. \]
			We show that these two statements are equivalent under the previous conditions. Indeed, we have 
			\begin{align*}
				d_3^q &= a_3^q(\alpha'/\alpha^{p^t})^{p^{2k-t}} \\
				\Leftrightarrow d_2^q \gamma_2^q &= a_2^q \gamma^q (\alpha'/\alpha^{p^t})^{p^{2k-t}} \\
				\Leftrightarrow d_2^q &= a_2^q \gamma_2^{q^3-q}(\alpha'/\alpha^{p^t})^{p^{2k-t}} = a_2^q \left((\alpha^{p^t}/\alpha')^{p^{m-t}}\right)^{q(q+1)} (\alpha'/\alpha^{p^t})^{p^{2k-t}} \\
				\Leftrightarrow d_2^q &= a_2^q (\alpha^{p^t}/\alpha')^{p^{k-t}}=a_2^q \alpha^{p^k}/\alpha'^{p^{k-t}}.
			\end{align*} 
			Substituting this condition into Eq.~\eqref{eq:t_condition} gives 
				\[(b'^{p^{m-t}}/b) = (\alpha'/\alpha^{p^t})^{p^{m-t}\cdot (q+1)}  (\alpha^{p^t}/\alpha')^{p^{m-t}\cdot (q+1)} = 1.\]
			We conclude that for fixed $\alpha,\alpha',t$ with $\alpha^{p^t}/\alpha'$ a $(q-1)$-st power, the presemifields $T(q,\alpha,1,b)$ and $T(q,\alpha',1,b')$ are isotopic if and only if $b=b'^{p^{t}}$ for some $0\leq t \leq m-1$. 	\\
			
			It remains to prove the statement on the centralizer. By Lemma~\ref{lem:yoshiara}, we only have to check autotopisms where the subfunctions of $L,M$ are monomials of degree $1$. This is a special case of this proposition, which is realized by setting $\P_1=\P_2$ and $t=0$. To compute the size of the centralizer, we have to go through our previous calculations and count the autotopisms of this form. For $a=0$, we have $a_3 = a_2 \gamma$, $d_3=d_2 \gamma^{\overline{q}^2}$ and $a_2^{q+1}=d_2^{q+1}$, where $\gamma$ is $(q-1)$-st root of unity (see Eqs.~\eqref{eq:t_q1} and~\eqref{eq:t_condition}), and all other coefficients are uniquely determined from that. So there are in total $p^m-1$ choices for $a_2$, $p^{\gcd(k,m)}-1$ choices of $\gamma$ and $\gcd(p^{k}+1,p^m-1)$ choices for $d_2$.
			
			For $a=1$, we have again $a_3 = a_2 \gamma$ and the additional conditions determine all other coefficients uniquely. So the centralizer has size $(p^m-1)(p^{\gcd(k,m)}-1)$ because there are again $p^m-1$ choices for $a_2$ and $p^{\gcd(k,m)}-1$ choices of $\gamma$. 
			\qed
\end{proof}

We are now able to piece everything together in the following result:

\begin{theorem} \label{thm:tani_isotopy}
Let $q_1=p^{k_1}$, $q_2=p^{k_2}$, 
\begin{align*}
\P_1&=T(q_1,\alpha,a,b)=(\F_{p^m} \times \F_{p^m},+,\circ_1), \textrm{ and},\\
\P_2&=T(q_2,\alpha',a',b')=(\F_{p^m} \times \F_{p^m},+,\circ_2)
\end{align*}
be Taniguchi pre-semifields such that 
$k_1 \neq m/2$ and $a,a' \in \{0,1\}$.
 If $k_2\equiv -k_1 \pmod m$ then, for fixed $\alpha',a',b'$, there exist $\alpha,a,b$ such that $\P_1$ and $\P_2$ are isotopic. 	If $k_1 \not\equiv k_2 \pmod m$ and $k_1\not\equiv -k_2 \pmod m$, the pre-semifields $\P_1$ and $\P_2$ are not isotopic.

	$T(q,\alpha,a,b)$ and $T(q,\alpha',a',b')$ are isotopic if and only if one of the two following cases occurs:
	\begin{itemize}
		\item $a=a'=0$, $\alpha^{p^t}/\alpha'$  is a $(q-1)$-st power in $\F_{p^m}^*$ and $b'^{p^t}/b$ is a $(q+1)$-st power in $\F_{p^m}^*$ for some $0\leq t \leq m-1$.
		\item $a=a'=1$, $\alpha^{p^t}/\alpha'$  is a $(q-1)$-st power in $\F_{p^m}^*$ and $b=b'^{p^t}$ for some $0\leq t \leq m-1$.
	\end{itemize}
\end{theorem}
\begin{proof}
Let us first check Condition~\eqref{eq:condition} from Theorem~\ref{thm:equivalence}. By Proposition~\ref{prop:2}, 	
\[
|C_{q_1,\alpha,a,b}|\in \{(p^{\gcd(k_1,m)}-1)(p^m-1) ,(p^{\gcd(k_1,m)}-1)(p^m-1)\cdot \gcd(p^m-1,p^{k_1}+1)\}. 
\]
Observe that $p'\nmid  (p^{\gcd(k_1,m)}-1)$ since $p'$ is a $p$-primitive divisor and $p'\nmid p^{k_1}+1$ since otherwise $p'|(p^{k_1}+1)(p^{k_1}-1)=p^{2k_1}-1$ which is not possible since $2k_1\neq m$ and, again, $p'$ is a $p$-primitive divisor. So $(p^m-1)p'\nmid |C_{q_1,\alpha,a,b}|$ and Condition~\eqref{eq:condition} is satisfied.

Assume $\P_1$ and $\P_2$ are isotopic. Then there is an isotopism $(N,L,M)$ between $\P_1$ and $\P_2$ with the properties stated in Theorem~\ref{thm:equivalence}. We already dealt with the case $k_1 \equiv -k_2 \pmod m$ in Proposition~\ref{prop:minus}, so we can assume $k_2< m/2$. The statement then follows from Proposition~\ref{prop:2}.\qed
\end{proof}
\section{Counting the number of non-isotopic Taniguchi semifields} \label{sec:counts}
To count the number of Taniguchi (pre-)semifields, we need a famous result by Bluher~\cite{Bluher} on projective polynomials and a well known basic lemma.

\begin{theorem}[{\cite[Theorem 5.6.]{Bluher}}] \label{bluher}
	Let $q=p^k$ and denote by $N(p,m)$ the number of polynomials $P(x)=x^{q+1}+x+b$ with $b \in \F_{p^m}$ such that $P$ does not have a root in $\F_{p^m}$. Set $d=\gcd(k,m)$ and $l=m/d$. Then
	\[N(p,m)=\begin{cases}
		\frac{p^{m+d}-p^d}{2(p^d+1)} & \text{ if } l \text{ is even,}\\
		\frac{p^{m+d}-1}{2(p^d+1)} & \text{ if } p,l \text{ are odd,}\\
		\frac{p^{m+d}+p^d}{2(p^d+1)} & \text{ if } p \text{ is even and } l \text{ is odd.}
	\end{cases}\]
\end{theorem}

\begin{lemma} \label{lem:gcd}
	Let $k,m \in \mathbb{N}$ and $p$ be a prime. Then
	\begin{itemize}
	\item $\gcd({p^k-1},{p^m-1}) = p^{\gcd({k},{m})}-1$.
	\item $\gcd({p^k+1},{p^m-1})=\begin{cases}
		1 & \text{if } m/\gcd({k},{m}) \text{ odd, and } p=2, \\
		2 & \text{if } m/\gcd({k},{m}) \text{ odd, and } p>2, \\
		p^{\gcd({k},{m})}+1 & \text{if } m/\gcd({k},{m}) \text{ even}.
	\end{cases}$
\end{itemize}
\end{lemma}

\begin{theorem}\label{thm_count}
	Let $N_T(p,k,m,a)$ be the number of non-isotopic Taniguchi semifields $T(q,\alpha,a,b)$ on $\F_{p^m} \times \F_{p^m}$ with $k \neq m/2$. Set $d=\gcd(k,m)$ and $l=m/d$. Then
	\[ (p^d-2) \cdot N(p,m)/m \leq N_T(p,k,m,1) \leq  (p^d-2) \cdot N(p,m),\]
	where $N(p,m)$ is determined in Theorem~\ref{bluher}. Further,
	\[(p^{d}-2)\cdot p^d/m \leq N_T(p,k,m,0)\leq (p^d-2)\cdot p^d\]
	if $l$ is even,
	\[(p^{d}-2)/m \leq N_T(p,k,m,0) \leq p^{d}-2\]
	if $p,l$ are odd and $N_T(p,k,m,0)=0$ if $p$ is even and $l$ is odd.  The total number of non-isotopic Taniguchi semifields with  $k\neq m/2$ is 
	\[N_T(p,m)=\sum_{k=1}^{\lfloor \frac{m}{2} \rfloor}\left(N_T(p,k,m,0)+N_T(p,k,m,1)\right).\]
\end{theorem}
\begin{proof}
	By Theorem~\ref{thm:tani_isotopy}, $T(q,\alpha,1,b)$ and $T(q,\alpha',1,b')$ are isotopic if and only if there is a $t$ such that $\alpha^{p^t}/\alpha'$ is a $(q-1)$-st power and $b'^{p^t}=b$. $\alpha^{p^t}/\alpha'$ is a $(q-1)$-st power if and only if $\alpha^{p^t},\alpha'$ are in the same coset of the cyclic subgroup with $(p^m-1)/\gcd(q-1,p^m-1)=(p^m-1)/(p^d-1)$ elements of $\F_{p^m}^*$.  There are thus $p^d-1$ such cosets. But $-\alpha,-\alpha'$ must not be $(q-1)$-st powers themselves by the necessary conditions for Taniguchi pre-semifields, so there are between $(p^d-2)/m$ and $p^d-2$ possible choices for $\alpha$ that yield non-isotopic pre-semifields. The overall number of permissible $b$ is (by Theorem~\ref{bluher}) $N(p,m)$.  For a fixed $b,t$ there is exactly one choice of $b'$ that yields an isotopic semifields, so there are $N(p,m)$ many non-isotopic choices for $b$, yielding the desired bound.
	
	For $a=0$, we have again between $(p^d-2)/m$ and $p^d-2$ choices for $\alpha$ for ranging $t$. $b,b'$ yield isotopic pre-semifields if and only if $b'^{p^t}/b$ is a $(q+1)$-st power for some $t$. Here, similar to before, this means that $b'^{p^t}$ and $b$ are in the same coset of the cyclic subgroup with $(p^m-1)/\gcd(q+1,p^m-1)$ elements of $\F_{p^m}^*$. By Lemma~\ref{lem:gcd}, we have $\gcd(q+1,p^m-1)=p^d+1$ if $l$ is even and $\gcd(q+1,p^m-1)=2$ if $l,p$ are odd and $\gcd(q+1,p^m-1)=1$ if $p=2$ and $l$ is odd. So the number of cosets is $p^d+1$, $2$ or $1$ depending on $p,l$. Since $-b,-b'$ themselves must not be $(q+1)$-st powers (by the conditions on the Taniguchi pre-semifield $x^{q+1}+b$ has no roots) we thus have $p^d$, $1$ or $0$ valid cosets. 
	
	By Theorem~\ref{thm:tani_isotopy}, different choices for $1 \leq k <m/2$ yield non-isotopic semifields, proving our result.
	\qed
\end{proof}

\begin{remark}
The precise values for $N_T(p,k,m,a)$ in Theorem~\ref{thm_count} depend on the precise values of $m$ and $d$, and could be computed with additional effort, see~\cite[Section 5]{kasperszhou} where a similar computation is applied for the $p=2$, $d=1$ case. However, these calculations are quite involved and since the factor $1/m$ that lies between the bounds in Theorem~\ref{thm_count} does not change the asymptotics of the result, we choose to not go into any more details.
\end{remark}

\section{Comparison with other semifield families and conclusion}

Theorem~\ref{thm_count} (together with Theorem~\ref{bluher}) shows that the total amount of pairwise non-isotopic Taniguchi semifields of order $p^{2m}$ is approximately $p^{m+s}$ where $s$ is the largest divisor of $m$, excluding $m/2$. In particular if $3|m$, the number of semifields is around $p^{\frac{4}{3}m}$. 

On the other hand, 
\begin{itemize}
\item the best known lower bound on the number of pairwise non-isotopic odd-order semifields of order $p^{n}$ constructed using skew-polynomial 
rings (or, equivalently, cyclic semifields; see~\cite{lavrauw2013semifields} for details) is less than $p^{n/2}$~\cite[Theorem 10]{johnson2009generalization}. 

\item The number of pairwise non-isotopic generalized twisted fields of order $p^{n}$ is around $p^t$, where $t$ is the largest divisor of $n$ that is less than $n/2$~\cite[Theorem 27]{purpura2007counting}. 

\item The best known lower bound on the number of pairwise non-isotopic semifields of order $p^{4l}$ constructed with the HMO construction is less than $p^{2l}$~\cite{kantor2009hmo}. We are not aware of any other construction of odd-order semifields that yields better lower bounds. 
\end{itemize}

The count in Theorem~\ref{thm_count} thus shows that the family of Taniguchi semifields yields a better bound compared to the best currently known lower bounds. In particular, the results in this paper show that the family of Taniguchi semifields is the largest known family of odd-order semifields to date. Interestingly, the amount of non-isotopic Taniguchi semifields is even larger than the known upper bound on the number of non-isotopic odd-order semifields of order $p^{n}$ constructed using skew-polynomial 
rings,  which is $p^{n/2}\log_2(p^n)$~\cite{kantor_liebler_2008}. Note that the family of semifields constructed via skew-polynomial rings was 
recently extended~\cite{sheekey2019}, however it remains so far unclear 
how much this changes the bound mentioned above. 
The upper bound given by Kantor~\cite[Theorem 1.6.]{kantor2009hmo} on the number of non-isotopic semifields from the HMO construction is larger than the number of non-isotopic Taniguchi semifields. However, it is unclear how strict this bound is.

Note that the situation for non-isotopic semifields of even order is quite different, 
as the construction by Kantor and Williams~\cite{KW} yields bounds that are much 
higher. More precisely, the number of non-isotopic semifields of even order is
{not bounded from below by a polynomial in the order} of the semifield (see~\cite{Kantor03}). 
Whether the number of non-isotopic semifields of odd order is also 
{not bounded by a polynomial in the order} of the semifield is an open problem, 
widely known as Kantor's conjecture. In fact, even the number of \emph{commutative} semifields of 
odd order is asymptotically not much different from the current state-of-the-art bound for 
non-commutative semifields, with~\cite{comm_arxiv} giving a family of around $p^{n/4}$ 
non-isotopic, commutative semifields of order $p^n$. It is thus desirable to construct 
new, larger families of {(non-commutative) semifields especially of odd order.}

\subsubsection*{Acknowledgements}
{The athors would like to thank William Kantor and Yue Zhou for their comments,
and bringing some literature on non-commutative semifields referred to in Conclusion
to their attention.}
We further thank an anonymous reviewer for spotting a calculation error in 
Proposition~\ref{prop:2}.

The first author is supported by {\sf GA\v{C}R Grant 18-19087S - 301-13/201843}. The second author is supported by NSF grant 2127742.

\subsubsection*{Data Availability Statement} Data sharing not applicable to this article as no datasets were generated or analysed
during the current study.

\subsubsection*{Conflicts of Interest}The authors have no conflicts of interest to declare that are relevant to the content of this article.

\bibliographystyle{amsplain}
\bibliography{semifields_arxiv} 
\end{document}